\documentclass{amsart}
\usepackage{graphicx}
\usepackage{amssymb}
\usepackage{array}
\usepackage{enumerate}
\usepackage{hyperref}

\newtheorem{thm}{Theorem}[section]
\newtheorem{lem}[thm]{Lemma}
\newtheorem{theorem theta}{Theorem}
\newtheorem{theorem main}{Theorem}

\theoremstyle{definition}

\theoremstyle{remark}

\numberwithin{equation}{section}

\begin{document}

\title[Computing character sums]{Computing Dirichlet character sums to a power-full modulus}
\author[G.A. Hiary]{Ghaith A. Hiary}
\thanks{Support by the Leverhulme Trust, and  
the National Science Foundation under agreements No.  DMS-0757627 (FRG grant) and DMS-0635607 (while at the Institute for Advanced Study).}
\address{Department of Mathematics, The Ohio State University, 100 Math Tower,
231 West 18th Ave, Columbus, OH 43210.}
\email{hiaryg@gmail.com}

\maketitle

\begin{abstract}
The Postnikov character formula is used to express large portions of a Dirichlet character sum in terms of quadratic exponential sums. 
The quadratic sums are then computed using an analytic algorithm previously derived by the author. This
leads to a power-saving if the modulus is smooth enough. 
As an application, a fast, and potentially practical, method to compute Dirichlet $L$-functions with complexity 
exponent 1/3 for smooth enough moduli is derived.
\end{abstract}

\section{Introduction}\label{sec: intro}

Fast algorithms for numerically evaluating 
 $L$-functions at individual points have generated some interest recently.
The asymptotic power-savings achieved by such algorithms 
have shed light on the basic nature of $L$-functions.
Also, given the wider availability of more powerful computers, 
there is better hope to reach regimes where such algorithms become practical, even if they are quite involved. 
One of these algorithms has already been used to compute the Riemann zeta function in 
small neighborhoods of many large values and to examine the zero distributions
there, see~\cite{H4,Bob} for examples of resulting data.
The fast algorithms for computing $L$-functions at individual points
can in some cases be combined with amortized-complexity (or \textit{average} cost) 
methods, such as \cite{H3},
to enable large-scale numerical studies at very large height and level, and in windows of considerable size, which   
is useful for testing random matrix theory predictions for $L$-functions.

New algorithms for computing the Riemann zeta function $\zeta(s)$ at individual points 
were derived in \cite{H1,H2}. One algorithm is capable of
numerically evaluating $\zeta(1/2+it)$ for $t>1$ with an error bounded by $O(t^{-\lambda})$ 
using $t^{1/3+o_{\lambda}(1)}$ bit operations. 
Another, faster, algorithm requires $t^{4/13+o_{\lambda}(1)}$ bit operations
and $t^{4/13+o_{\lambda}(1)}$ bits of storage.
There are interesting analogies between these algorithms,
which rely on fast computation of exponential sums, 
and methods for proving subconvexity estimates for zeta,
meaning bounds of the form $\mu_{\zeta}(1/2) \le 1/4 - \delta'$, 
for some $\delta'>0$, where $\mu_{\zeta}(1/2)$ is the infimum of all the numbers $\eta$ 
such that $\zeta(1/2+it) = O(|t|^{\eta})$. For example, both  
 the $t^{1/3+o_{\lambda}(1)}$ algorithm and
the Weyl-Hardy-Littlewood method (which yields the bound
$\mu_{\zeta}(1/2) \le 1/6$; see \cite[5.3]{T}) 
start by subdividing the ``main sum'' of zeta in a similar way.
In both cases, the main difficulty is reduced to understanding quadratic exponential 
sums $K^{-j} \sum_{0\le k<K} k^j e^{2\pi i \alpha k + 2\pi i \beta k^2}$,
where $\alpha,\beta \in [0,1)$. But unlike 
the Weyl-Hardy-Littlewood method, which detects cancellation in the quadratic sum 
due to the linear argument $\alpha$ solely, the iterative
algorithm for computing quadratic sums, which will play an 
essential role in this paper, relies in a crucial way on 
controlling the length of the sum via the quadratic argument
$\beta$. This algorithm, which is derived in \cite{H1}, 
repeatedly applies van der Corput methods, together with
an intervention to normalize $\alpha$ and $\beta$ suitably, and in a process
reminiscent of the method of exponent pairs of van der Corput and Phillips
(see \cite[5.20]{T}). The effect is to shorten the length of the exponential sum
with each iteration by a factor of $1/2$ or better so that the algorithm
finishes using $O(\log K)$ iterations.

Van der Corput methods were also the inspiration behind the $t^{4/13+o_{\lambda}(1)}$ algorithm,
that time leading to cubic 
sums $K^{-j} \sum_{0\le k<K} k^j e^{2\pi i \alpha k + 2\pi i \beta k^2 + 2\pi i
\gamma k^3}$,
where the situation was much more complicated (even though the range of $\gamma$ happened to be very restricted in application).
The $t^{4/13+o_{\lambda}(1)}$ algorithm relies, in addition, on a fast Fourier transform (FFT) precomputation, which resembles
the precomputation in Sch\"onhage's method~\cite{S} for computing zeta in
$t^{3/8+o_{\lambda}(1)}$ time.
In light of the similarities pointed so far between fast algorithms and subconvexity estimates for
zeta, 
it might be of interest to try to understand the complexity exponent $4/13$ in the subconvexity context.

In the $\textrm{GL}_2$ setting, Vishe has derived an algorithm in \cite{V1}
that permits the accurate computation of 
$L(1/2+it,\tilde{f})$, where $\tilde{f}$ is a fixed 
 modular cusp form (so big-$O$ constants depend on $\tilde{f}$), 
 holomorphic or Maass,
of weight $k$ for a congruence subgroup of $\textrm{SL}_2(\mathbb{Z})$, 
using $O(1+t^{7/8+\epsilon'})$ operations, where 
$t>0$ and $\epsilon'$ is any fixed positive real.
The method employed in \cite{V1}
is described as geometric, and is closely related to subconvexity
estimates for $\textrm{GL}_2$ $L$-functions.  
It ultimately uses a direct precomputation (not involving the FFT) to obtain its power-savings. 
The method relies on an integral representation of $L(s,\tilde{f})$, and has the useful feature 
that only the first few coefficients of the cusp form are needed for computing the corresponding
$L$-function. Vishe has also derived an algorithm in \cite{V2}
for computing the central value $L(1/2,\tilde{f}_1 \times\chi)$, where $\chi$ is a character $\textrm{mod}\,\,q$
and $\tilde{f}_1$ is a fixed cusp form of weight $k$ for the full modular group,
with complexity as good as $q^{5/6+o(1)}$ when $q$ is highly composite.
His method is closely related to recent work of Venkatesh on 
subconvexity estimates for $L(1/2,\tilde{f}_1 \times\chi)$.

The goal of this article is to extend the investigations in \cite{H1,H2} 
 to the so-called $q$-aspect.
We derive a potentially practical algorithm for computing $L(s,\chi)$, 
where $\chi$ is a character $\textrm{mod}\,\,q$, with an error bounded
by $O(q^{-\lambda}(|s|+1)^{-\lambda})$ using
$q^{1/3+o_{\lambda}(1)}(|s|+1)^{1/3+o_{\lambda}(1)}$ operations
when $q$ is smooth enough, where smooth enough means 
 that the radical of $q$ is small enough compared with $q$, or 
$q$ is power-full. 
Specifically, we prove the following upper bound on the number of operations 
 required for computing $L(s,\chi)$. (Our computational model is detailed 
 in \textsection{\ref{comp model}},
 including the meaning of \textit{operation} and how to input the character
 $\chi$ to the algorithm.)
\begin{thm}\label{mainthm}
There are absolute constants $A_1,\ldots,A_5$,
$\kappa_1$, $\kappa_2$, and $\kappa_3$, such that for any real number $\lambda$, 
any complex number $s$ with  $1/2 \le \Re{s} \le 1$, any positive integer 
$q = p_1^{a_1}\cdots p_h^{a_h}$ (where $p_j$ are distinct primes), 
and any given character $\chi \bmod{q}$,
the value of the Dirichlet $L$-function $L(s, \chi)$ can be computed 
to within $\pm\, q^{-\lambda} (|s|+1)^{-\lambda}$ using 
$\le A_1\, p_1^{\lceil a_1/3\rceil}\cdots p_h^{\lceil a_h/3\rceil} (|s|+1)^{1/3} \,(\lambda+1)^{\kappa_1}\, \log^{\kappa_1}(q(|s|+1))$ operations on numbers of $\le A_2\, (\lambda+1)^4\, \log^4 (q(|s|+1))$ bits, provided a precomputation, that depends on 
$q$ only, costing
$\le A_3 \,(p_1+\cdots +p_h)\, \log^{\kappa_2} q$ operations on 
numbers of $\le A_4\, \log q$ bits, and requiring $\le A_5 \,(p_1+\cdots +p_h) \, \log^{\kappa_3} q$  bits of storage, is performed.
\end{thm}
The precomputation in Theorem~\ref{mainthm} comes directly from lemma~\ref{lem:compC}, which furnishes a procedure 
for computing individual values of $\chi$ to within $\pm \epsilon$ in 
poly-log time (in $q$ and $1/\epsilon$) using precomputed values.
Theorem~\ref{mainthm} assumes that the factorization of $q$ is given,
but this is not essential since there are algorithms with provable
complexity for factoring $q$ at a cost that is subsumed by the overall cost of
our algorithm (e.g. Lehman's method; see \cite{CP}).
The condition $1/2 \le \Re{s} \le 1$ in the theorem is not essential,
and is inserted partly to simplify dealing with the $L$-functions associated with 
imprimitive characters; see the discussion preceding the proof of Theorem~\ref{mainthm} in \textsection{\ref{sec: app}}. 
The upper bound $O_{\lambda}(\log^4(q(|s|+1)))$
for the number of bits is generous, and is to simplify various proofs; 
it comes directly from Theorem~\ref{lem:l2}, and all that is 
required otherwise is $O_{\lambda}(\log(q(|s|+1)))$ bit arithmetic.
The algorithm can be modified easily so that arithmetic is wholly performed using $O_{\lambda}(\log (q(|s|+1)))$ 
bits, which is what one should do in a practical version. 
The algorithm applies uniformly in $q$ and $s$, 
and represents the first power-saving in the $q$-aspect 
over previous algorithms (which consume $\gg q^{1/2+o_{\lambda}(1)}$ time) 
for any class of Dirichlet $L$-functions. 

Part of the story of the algorithm in Theorem~\ref{mainthm}
is the general analogy between the $t$-aspect and 
the depth aspect (highly power-full moduli such as $q=p^a$, $a\to \infty$) 
in the theory of character sums and Dirichlet $L$-functions.
Indeed, the power-full structure of the modulus will play an important role
in the algorithm. We note, though, that it is not necessary for the exponent $a$ to be large in order for the algorithm
to perform its best. The running time  
$q^{1/3+o_{\lambda}(1)}(|s|+1)^{1/3+o_{\lambda}(1)}$ is still achieved
if simply $q=p^a$ and $3 \mid a$ (for example, if $q=p^3$), 
and, more generally,  if the exponents of the prime factors of $q$ are divisible by 3.
The algorithm does not provide new savings when $q$ is 
square-free (e.g.\ if $\chi$ is a real primitive character) 
since a direct application of it in this case requires about $q(|s|+1)^{1/3+o_{\lambda}(1)}$ time. 

Our algorithm is related to proofs of subconvexity estimates for $L(1/2+it,\chi)$, 
which further strengthens the apparent connection 
between algorithms and subconvexity estimates for $L$-functions.
Indeed,  one of the essential steps in deriving the algorithm 
is a specialization of the Postnikov character formula, stated in lemma~\ref{lem:l1}, 
where we exploit the power-full structure of the modulus $q$.
Postnikov's formula was employed 
by Barban, Linnik, and Tshudakov in \cite{BLT}
 to study the same family of Dirichlet $L$-functions tackled here. 
 They proved the estimate $\sum_{n\le N} \chi(n) \ll \sqrt{N} q^{1/6} (\log q)^{1/2}$, where $q=p^n$,
$p \ge 3$ is any fixed prime, $n > n_0$, $ N \le q^{2/3}$, and $\chi\bmod{q}$ is a 
non-principal character\footnote{It was mentioned in \cite[lemma 6]{BLT} that the estimate $\sum_{n\le N} \chi(n) \ll \sqrt{N} q^{1/6} (\log q)^{1/2}$
 would still hold for $N > q^{2/3}$ since it would be a consequence of the P\'olya-Vinogradov inequality. But this  
 seems to miss an extra factor of $(\log q)^{1/2}$, which in turn impacts
 the bound for $|L(1/2+it,\chi)$ by an extra factor of $(\log q)^{1/2}$. 
 Nevertheless, we stated the bounds here the same as in \cite{BLT}.},
 from which the estimate $|L(1/2+it,\chi)| \ll (|t|+1) q^{1/6} (\log q)^{3/2}$ was deduced.

The connection between algorithms and subconvexity estimates 
is somewhat puzzling, especially since, a priori, there is no compelling reason
for it  to exist. 
Fast algorithms work because 
they are able to express the $L$-function using fewer ``terms'' (whose 
sizes matter on a \textit{logarithmic scale} only),  whereas
subconvexity estimates rely on detecting cancellation 
among terms, and involve certain critical steps that are too crude for computation, 
or for which there is no simple computational analogue, such as elementary applications of 
the Cauchy-Schwarz inequality. Also, despite the parallels between the $t^{1/3+o_{\lambda}(1)}$ algorithm 
and, both, the Weyl-Hardy-Littlewood method and the method of exponent pairs, 
the $t^{1/3+o_{\lambda}(1)}$ algorithm still does not yield the bound $\mu_{\zeta}(1/2) \le 1/6$. 
This is because the algorithm
does not guarantee any cancellation should occur in the quadratic sum as it
repeatedly applies van der Corput methods to it. The only way the algorithm
could sense (or roughly distinguish) the size of the quadratic sum is via the total number
of iterations that it uses in the computation. But this number is only poly-log in the length of the sum, 
and, therefore, variations in it do not affect the power-savings. 
Furthermore, consider that the main difficulty in improving the $t^{4/13+o_{\lambda}(1)}$ 
algorithm is not that certain terms 
are getting too large, but that a certain FFT precomputation, for 
which  there is no clear analogue when bounding $\mu_{\zeta}(1/2)$, becomes too 
expensive. 

We note that the algorithm for computing 
$L(s,\chi)$, which is stated in Theorem~\ref{mainthm}, actually relies on
a character sum estimate, namely the P\'olya-Vinogradov inequality, 
to obtain an upper bound for $|\sum_{n\ge M} \chi(n) n^{-s}|$. 
This reduces the computation of $L(s,\chi)$ 
to computing a main sum $\sum_{n<M} \chi(n)n^{-s}$, where $M$ 
is chosen according to the desired precision; see \eqref{eq:lchierr}.  
However, our use of the P\'olya-Vinogradov inequality is not essential because,
as the discussion in \textsection{\ref{sec: app}} shows, 
the P\'olya-Vinogradov inequality  can be replaced 
at a small loss by the trivial estimate $|\sum_{n<N} \chi(n)|< q$.
In general, the available character sum estimates 
help us to obtain a shorter main sum for $L(1/2+it,\chi)$,
but not short enough to improve the algorithmic power-savings.
It might be worth mentioning, however, that there are examples of algorithms 
that achieve their saving by directly using a character sum estimate
to obtain a relatively short main sum when $\Re(s)$ is large enough. 
For instance, the algorithm of Booker~\cite{Boo}, which can certify 
the output of Buchmann's conditional algorithm for computing the class number of a quadratic 
number field $\mathbb{Q}(\sqrt{d})$ in time $|d|^{1/4+\epsilon'}$ if the output of Buchmann's algorithm
is correct (as expected), and in time $|d|^{1/2+\epsilon'}$ otherwise, 
directly uses Burgess' theorem (see \cite{GL}) 
to obtain a bound on the truncation error in a smoothed approximate functional
equation for $L(1,\chi)$.
This bound is then shown to suffice considering that the class number needs to be 
computed to within half an integer only.

In the remainder of the introduction, we overview the structure of the paper.
In \textsection{\ref{comp model}}, we discuss our underlying computational
model.
In \textsection{\ref{sec:afe}}, we provide background material on some previous methods, 
which will help place the algorithmic improvements obtained here into context.
We also provide a sketch of our method for computing the character sums 
\begin{equation}
S_{\chi}(K):=\sum_{0\le k <K} \chi(k)\,,
\end{equation} 
which is the main step towards the algorithm for $L(s,\chi)$. 
Notice that, while $S_{\chi}(K)$ and $L(s,\chi)$ are directly related, 
for example via $L(s,\chi) = s\int_1^{\infty} S_{\chi}(x)x^{-s-1}\,dx$,
$\Re(s)>0$,
it is not immediate (or necessary) that savings achieved in computing $S_{\chi}(K)$
must translate to savings in computing $L(s,\chi)$.
In our case, though, the method for computing $S_{\chi}(K)$ generalizes naturally to $L(s,\chi)$.
In \textsection{\ref{sec:compS}}, we  state
and prove several needed results, starting with the simpler case of $S_{\chi}(K)$.
Then, as an application, we prove in \textsection{\ref{sec: app}} the complexity of 
the algorithm for $L(s,\chi)$ in Theorem~\ref{mainthm}.
Last, in \textsection{\ref{sec:general sum}}, we remark on the general modulus case.

\section{Computational model and notation} \label{comp model}
We now specify our computational model.
Real numbers are represented using a fixed point system in base-2. 
So when we write that operations are performed using $\mathcal{B}$-bit 
arithmetic, it means that real numbers are represented using $\mathcal{B}$ bits
to the left of the radix point (integer bits), and $\mathcal{B}$ bits to the right of the 
radix point (fraction bits), together with a sign bit (specifying 
whether the number is positive or negative). The position of the radix point is fixed.
This system can accommodate integers easily by removing the fraction bits entirely, and 
requiring an extra bit (flag) to tell whether the number is an integer.
An integer can be coerced into a real number by padding it with all zero fraction bits
then flipping the integer flag, and vice versa.
We can represent a real number $x$ using a $\mathcal{B}$-bit fixed system if 
$\log |x| /\log 2 < \mathcal{B}$, otherwise an overflow problem could occur. 
Notice that, if there is no overflow, then $x$ can be represented with a
round-off error of $\pm 2^{-\mathcal{B}}$.
Similarly, to represent an integer $m$ exactly, we need $\log |m| /\log 2 < \mathcal{B}$. 
Our theorems and lemmas 
will always request a large enough value of $\mathcal{B}$ to guarantee no overflow  
ever occurs. So, when we write that it suffices to use $\mathcal{B}$-bit 
arithmetic in a certain algorithm, it implicitly means that 
$\log \mathcal{M} /\log 2 < \mathcal{B}$,
where $\mathcal{M}$ is the largest number 
that ever occurs in the algorithm (regardless of the order of operations). 
In our algorithm, the number $\mathcal{M}$ depends on $s$ and $q$ only. It
will be implicitly tracked during the derivations, and will manifestly satisfy a bound of the form
 $0\le \log\mathcal{M} < \tilde{A}_0((\lambda+1)\log(q(|s|+1)))^4$ for some absolute constant $\tilde{A}_0$,
 which is generous.

Our basic real operations 
are addition, multiplication, division by a non zero number, the cosine and the sine, exponentiation, and taking the logarithm of a positive number. 
We take for granted that there are algorithms to perform each of these operations 
in the $\mathcal{B}$-bit fixed point system using 
$\le \tilde{A}_1\mathcal{B}^{\tilde{\kappa}_1}$ bit operations and $\le \tilde{A}_2\mathcal{B}^{\tilde{\kappa}_2}$
bits of storage, where $\tilde{A}_1$, $\tilde{A}_2$, $\tilde{\kappa}_1$, and $\tilde{\kappa}_2$, 
are absolute numbers, and such that if no overflow occurs then
the final output of the operation is correct to within
$\pm 2^{-\mathcal{B}+\mathcal{F}}$, where $\mathcal{F}$ is an
absolute number.
Let $\mathcal{C}$ be an upper bound on the total number of basic real operations
consumed by a given algorithm. 
Then, an upper bound for the round-off error accumulated during a full run of the algorithm is $\mathcal{C} \mathcal{M} 2^{-\mathcal{B}+\mathcal{F}}$. 
So, if the final output is requested with an error tolerance $\epsilon \in (0, e^{-1})$,
then it suffices to ensure that $\mathcal{C} \mathcal{M} 2^{-\mathcal{B}+\mathcal{F}} < \epsilon$, 
or $(\log \mathcal{M} + \log\mathcal{C} + \log(1/\epsilon))/\log 2 + \mathcal{F} < \mathcal{B}$,
which is one way to determine upper bounds for the required number of bits in our algorithms.

We make a similar assumption about the existence of algorithms to perform 
the basic integer operations: Addition, multiplication, 
and checking equality of two integers,  
such that operations can be performed exactly in the 
fixed point system, provided there is no overflow problem.
(Division of integers is carried out using real numbers.)
We assume that the elements of the ring $\mathbb{Z}/p^a \mathbb{Z}$ are modelled by (or identified with) 
the set of numbers $\{0,\ldots,p^a-1\}$, in the obvious way. The basic ring operations are 
addition, multiplication, and checking equality of two elements. 
Determining the multiplicative inverse, if it exists, of an element in $\mathbb{Z}/p^a \mathbb{Z}$
can be done via the Euclidean algorithm, which is fast.
Ring operations are performed exactly using integer arithmetic in our system 
and by reducing modulo $p^a$.

Our algorithms will sometimes request to perform a precomputation, 
and to store the output in main memory for later retrieval.  
We assume that any randomly-chosen precomputed value can be quickly retrieved
from main memory
in roughly the same amount of time. 
This is usually a realistic assumption in core memory areas, and is one feature of the 
Random Access Machine model in complexity analysis, as mentioned in \cite{LMO} for example.
(This assumption is actually not essential for Theorem~\ref{mainthm}, but this is not important.) 
For the sake of definiteness, let us suppose that if 
a precomputation results in $\mathcal{T}$ numbers say, where
each number is represented using a $\mathcal{B}$-bit fixed point
system, then the cost of retrieving a precomputed number
is uniformly bounded by $\le
\tilde{A}_3(\mathcal{B}+\log\mathcal{T})^{\tilde{\kappa}_3}$ bit operations,
where $\tilde{A}_3$ and $\tilde{\kappa}_3$ are absolute constants. 

The computational complexity of the algorithms here is measured by 
the number of integer, real, and ring operations (or simply, operations) consumed.
This in turn can be routinely bounded in terms of bit operations
since all the numbers that occur in Theorem~\ref{mainthm} 
can be expressed using a poly-log number of bits in $q$ and $|s|$.
We will specify what it means for a character $\chi$ to be ``given'' 
as an input to algorithms at the beginning of \textsection{\ref{sec:compS}}.
When we write: ``$L(s,\chi)$ can be computed to within $\pm\,q^{-\lambda} (|s|+1)^{-\lambda}$,'' it
means that, given $s$ and $\chi$, we can find a number $\tilde{L}(s,\chi)$
such that $L(s,\chi) - \tilde{L}(s,\chi) = e^{i\omega_0} \epsilon_0$ for some unknowns $\omega_0 \in \mathbb{R}$ 
and  $-q^{-\lambda} (|s|+1)^{-\lambda} \le \epsilon_0 \le q^{-\lambda} (|s|+1)^{-\lambda}$.
So, in particular, $|L(s,\chi) - \tilde{L}(s,\chi)| \le q^{-\lambda} (|s|+1)^{-\lambda}$.

Last, we remark that in practice one typically uses a floating point system for representing numbers, 
not a fixed point system. This said, the asymptotic power-savings of our algorithms are independent 
of which system is used. The reason we prefer a fixed point system in the exposition is 
to make it fairly simple to determine the number of bits needed to perform basic operations. 
Also, in practice, it is worthwhile to minimize the use of expensive multi-precision 
arithmetic by carefully using Taylor expansions, 
precomputations, suitable orders of operations, and various standard tricks. 

\vspace{2mm}

\noindent
\textit{Notation.}  We have been using asymptotic notation. For completeness, 
let us define it explicitly. Following \cite[p. xiii]{D}, we write $f_2(x) = O(f_3(x))$, 
or equivalently $f_2(x) \ll f_3(x)$, when there is an absolute constant
$C_1$ such that $|f_2(x)| \le C_1 f_3(x)$ for all values
of $x$ under consideration. In this paper, the ``values of $x$ under consideration''
is always a set of the form $x \ge C_2$, where $C_2$ is an absolute constant. 
We write $f_2(x) = o(f_3(x))$ when
$\lim f_2(x)/f_3(x) = 0$, where the limit is always taken as $x\to \infty$
in this paper.
When we write $O_{\tilde{\lambda}}(.)$
or $o_{\tilde{\lambda}}(1)$, it means the implied constants 
depend on $\tilde{\lambda}$.
So, for example, when we write $O_{\tilde{\lambda}}(.)$, 
it mean that $C_1$ and $C_2$ depend on $\tilde{\lambda}$. 
If no dependence is indicated, then the implied
constants are absolute (but for more emphasis, we will frequently state
this explicitly). We will often refer to poly-log factor 
in some (positive) parameters $x_1,\ldots,x_{r'}$, which means
a factor of the form $A (\log(x_1+3) + \cdots + \log(x_{r'}+3))^{\kappa}$, 
where $A$ and $\kappa$ are absolute constants.

\section{Background and the basic idea} \label{sec:afe}

In \cite[Theorem 1]{BF}, Bombieri and Friedlander prove, 
for a fairly large class of $L$-functions, that it is not possible to approximate  
a fixed $L$-function $L(s)=\sum_{n=1}^{\infty} c_n n^{-s}$ 
(so big-$O$ constants will depend on $L$),
in a window $T\le \Im(s)\le 2T$ and just to the left of the critical line, 
with an error $O(T^{-\epsilon_0})$, $\epsilon_0>0$, 
using a \textit{single} Dirichlet 
polynomial $\sum_{n<x} c_n(x)n^{-s}$, $|c_1(x)|>1/2$, $c_n(x)\ll n^{o(1)}$,
of length much shorter than the ``analytic conductor'' of $L(s)$.
Note that the analytic conductor terminology usually appears 
in connection with subconvexity estimates, but it also arises naturally in 
connection with algorithms for $L$-functions, such as those in \cite{R1}.
Following \cite{IK}, the analytic conductor of $\zeta(s)$ can be defined
as $|s|+3$, and the analytic conductor of $L(s,\chi)$ 
can be defined as $q (|s+\mathfrak{a}|+3)$, where 
$\mathfrak{a} := (1-\chi(-1))/2$. The precise definition 
is not important to the asymptotic results discussed here, and the reader
may refer to \cite{IK} for the definition for a general type of $L$-function. 

The result of Bombieri and Friedlander assures, for example, that
$L(\sigma+it,\chi)$, $\sigma<1/2$, $T\le t \le 2T$, 
can not  be approximated well by a single Dirichlet polynomial
of length $T^{1-o(1)}$.
However, as remarked in \cite{BF}, the behavior in a fixed strip to the right of the critical
line is quite different: On the Lindel\"of hypothesis, $L(s)$ can be approximated 
there with an error $o(1)$ using arbitrarily short Dirichlet polynomials. Therefore,
if we define $\tilde{\mu}_L(\sigma)$, $0\le \sigma \le 1$ say, as
the infimum of all the numbers $\eta$ such that $L(\sigma+it)$ 
can always be approximated in $T\le t \le 2T$ with an error $o(1)$ 
using a Dirichlet polynomial $\sum_{n<x} c_n(x) n^{-s}$
of length $x = O(T^{\eta})$, then it is clear that $\tilde{\mu}_L(\sigma)$
and $\mu_L(\sigma)$ are qualitatively different. 
For example, $\mu_L(\sigma)$ is continuous,
whereas $\tilde{\mu}_L(\sigma)$ is not (assuming the Lindel\"of hypothesis for $L(s)$).

If the restriction on the number of Dirichlet 
polynomials is removed, then one can do better. For example, it is possible 
to approximate many $L$-functions with an error $o(1)$ using two Dirichlet 
polynomials, each of length roughly the square-root of the analytic conductor. 
A general method for doing so is 
the smoothed approximate functional equation, which we will discuss
shortly in more detail. 
But, first, we remark that in all known algorithms where the square-root barrier is
broken, one had to ultimately abandon the framework of Dirichlet polynomials.
For instance, the algorithms in \cite{H2} rely on
approximations via low-degree exponential sums, and so 
does the algorithm of Theorem~\ref{mainthm}; see also \cite{S,V1,V2}. 

While Theorem~\ref{mainthm} never uses the smoothed approximate functional equation, 
it is still useful to discuss such a general method here since it supplies formulas
for computing many $L$-functions, including $L(s,\chi)$. 
Formula \eqref{eq:smoothed approx} below is actually a specialization 
of a smoothed approximate functional equation formula in \cite{R1}, 
 valid for a Dirichlet series with arbitrary coefficients provided the series possesses
a meromorphic continuation and a functional equation, and 
satisfies very mild growth conditions (so no Euler-product is required).
See \cite{R2} for a \verb!C++! implementation. 

Specifically, Rubinstein \cite{R1} provides two formulas for $L(s,\chi)$,
for the even and odd  cases,
which can be combined straightforwardly as follows:
Assume $\chi\bmod{q}$ is a primitive character, and let
$\mathfrak{a} := (1-\chi(-1))/2$, then

\begin{equation} \label{eq:smoothed approx}
\begin{split}
\left(\frac{q}{\pi}\right)^{\frac{s+\mathfrak{a}}{2}}\Gamma\left(\frac{s+\mathfrak{a}}{2}\right)L(s,\chi)\delta^{-s} &= \delta^{\mathfrak{a}} \sum_{n=1}^{\infty} \chi(n) G\left(\frac{s+\mathfrak{a}}{2},\frac{\pi n^2\delta^2}{q}\right) \\
&\quad + \frac{1}{\delta^{\mathfrak{a}+1}}\frac{\tau(\chi)}{i^{\mathfrak{a}} q^{1/2}}\sum_{n=1}^{\infty} \overline{\chi(n)} G\left(\frac{1-s+\mathfrak{a}}{2}, \frac{\pi n^2}{\delta^2 q}\right)\,,
\end{split}
\end{equation}
where $G$ is a smoothing function, expressed in terms of
the incomplete Gamma function $\Gamma(z,w)$,
\begin{equation}
G(z,w) := w^{-z}\Gamma(z,w) = \int_1^{\infty} e^{-wx}x^{z-1}\,dx\,,\qquad \Re(w) > 0\,,
\end{equation}
$\tau(\chi)$ is the usual Gauss sum, 
\begin{equation}
\tau(\chi) := \sum_{n=1}^q \chi(n)e^{2\pi i n/q}\,,
\end{equation}
and $\delta$ is a certain complex parameter, with a simple dependence on $s$, chosen to
cancel out the exponential decay in $\Gamma((s+\mathfrak{a})/2)$ as $|\Im(s)| \to \infty$; 
see~\cite{R1} for details. 
Although the series in \eqref{eq:smoothed approx} are infinite, the weights
$G(z,w)$ decay exponentially fast when $\Re(w) \gg 1$. For a given $\lambda$,
 the series \eqref{eq:smoothed approx} can be truncated after 
$q^{1/2+o_{\lambda}(1)}(|s|+1)^{1/2+o_{\lambda}(1)}$ terms
with a truncation error $\pm q^{-\lambda}(|s|+1)^{-\lambda}$.
Once truncated, the series can be evaluated term by term to give 
a numerical approximation of $L(s,\chi)$ accurate 
to within $\pm 2q^{-\lambda}(|s|+1)^{-\lambda}$ say. 
Therefore, the number of terms needed is roughly equal to the square-root of the 
analytic conductor of $L(s,\chi)$. 
Notice that formula \eqref{eq:smoothed approx} involves the evaluation of
the Gauss sum $\tau(\chi)$, which requires summing an additional $q$ terms when done in a straightforward way.
Also, a direct application of \eqref{eq:smoothed approx} 
requires computing roughly the first ``square-root of the analytic conductor''
Dirichlet coefficients.

In the case of $S_{\chi}(K)$, where $\chi$ is primitive, one can use the multiplicativity of $\chi$, 
together with a suitable choice of a smoothing function, to always (regardless of $K$) express $S_{\chi}(K)$
as a sum involving $q^{1/2 + o_{\lambda}(1)}$ terms multiplied by $\tau(\chi)$; 
see \textsection{\ref{sec:general sum}}.
If $K$ is smaller than $q^{1/2}$, however, then such a series 
does not lead to a faster computation 
since it is longer than the original sum.
In Theorem~\ref{lem:l2}, we provide a different method for computing $S_{\chi}(K)$, 
which leads to asymptotic speed-ups if $q$ is smooth enough
(and which, in turn, is the main ingredient in the proof of Theorem~\ref{mainthm}).

We sketch Theorem~\ref{lem:l2}. 
Let $\chi \bmod{q}$ be any character, where $q = p_1^{a_1}\cdots p_h^{a_h}$ 
($\chi$ need not be primitive). 
Theorem~\ref{lem:l2} assures that $S_{\chi}(K)$ can be computed to within $\pm \epsilon$
in about $p_1^{\lceil a_1/3\rceil}\cdots p_h^{\lceil a_h/3\rceil}$ time, 
up to a poly-log factor in $q$ and $1/\epsilon$. 
This running-time improves on $q^{1/2}$ for many choices of the $a_j$, or roughly when

\begin{eqnarray}\label{eq:time req}
\prod_{a_j\in\{1,2,4\}} p_j^{\lceil a_j/3\rceil - a_j/2} q^{\epsilon_1} \ll \prod_{a_j\not\in\{1,2,4\}} p_j^{a_j/2-\lceil a_j/3 \rceil} \,,
\end{eqnarray} 
for some $\epsilon_1 > 0$.
Notice that $\lceil a_j/3\rceil - a_j/2 = 1/2$ if $a_j=1$, and it vanishes if $a_j\in\{2,4\}$,
so the l.h.s. is simply $\prod_{a_j=1} \sqrt{p_j}$, which is the square-root of the square-free part of $q$.
The behavior of the algorithm of Theorem~\ref{lem:l2} as $q\to \infty$ is very well-controlled, in the sense that power-savings are
obtained regardless of whether $q\to \infty$ through some of the $a_j$'s or some of 
the $p_j$'s or any combination thereof.
For example, if $q=p^{3a}$, then the method requires about $p^a=q^{1/3}$ time (even if $p=5$ say). 
As another example, if $q=p_1 p_2^{3a}$, then the time requirement is about 
$p_1 p_2^a$, which represents a power-saving beyond $q^{1/2}$ when $p_2^a \gg p_1 q^{\epsilon_2}$
for some $\epsilon_2>0$.
Roughly speaking, if the $a_j$ are large, or if the $p_j$ are large but with exponents nearly divisible by 3, then 
the running time is about $q^{1/3+o(1)}$.

Since our methods exploit the power-full structure of the
modulus (via the Postnikov character formula), 
it is not surprising that $a_j=1$, which corresponds to the prime modulus case, 
appears as an exceptional case 
in \eqref{eq:time req}, meaning it is a case where we do not improve on $q^{1/2}$. 
But the appearance of $a_j=2$ and $a_j=4$ as exceptional cases 
is somewhat surprising. 
The reason we do not obtain a power-saving beyond $q^{1/2}$ when $a_j=4$, for example, is because we encounter 
 cubic exponential sums with possibly large cubic coefficients. 
 There is no general algorithm to 
compute such sums faster than required by a straightforward evaluation 
except for the algorithm of~\cite{H2}, which is suitable for sums with small cubic coefficients.

We illustrate the basic idea of Theorem~\ref{lem:l2} in the situation $q=p^a$
and simplifying to a character sum.
For further simplicity, assume $K$ is a multiple of $p^{\lceil a/3\rceil}$. So,
\begin{eqnarray}\label{eq:sketch}
S_{\chi}(K)=\sum_{\substack{0<l<p^{\lceil a/3\rceil}\\ (l,p)=1}} \chi(l) \sum_{0\le k<K/p^{\lceil a/3 \rceil}} \chi(1+ k\overline{l} p^{\lceil a/3\rceil})\,,
\end{eqnarray}  
where $(m,n)$ denotes the greatest common divisor of $m$ and $n$, and $\overline{l}$ 
is determined by the relation $l \overline{l} \equiv 1 \bmod{p^a}$. 
The particular choice of the exponent $\lceil a/3\rceil$ in \eqref{eq:sketch} 
is so that the inner sum there can be expressed as a quadratic exponential sum 
via a specialization of the Postnikov character formula in lemma~\ref{lem:l1}.
Once expressed this way, the inner 
sums can be computed to within $\pm \epsilon$ 
in poly-log time (in $q$ and $1/\epsilon$) 
using the algorithm in \cite[Theorem 1.1]{H1}; see Theorem~\ref{thetaalg} here.
If the exponent $\lceil a/3\rceil$ in \eqref{eq:sketch} is 
decreased any further, then, in general, the Postnikov 
character formula yields cubic and higher degree exponential sums. 
In Theorem~\ref{lem:l2}, this idea is generalized to quadratic sums twisted by $\chi$. 

\section{Computing $S_{\chi}(K)$}\label{sec:compS}

Let $q = p_1^{a_1}\cdots p_h^{a_h}$.
(We assume the prime factorization of $q$ is given to us, but 
factoring $q$ does not cause a complexity
issue in any case; for example, Lehman's method, see \cite{CP}, can be used to factor $q$
in $q^{1/3+o(1)}$ time, and with no storage space requirement.)
Let $\chi$ be a character $\textrm{mod}\,\,q$.
We first discuss how $\chi$ should be ``given'' as an input to the algorithms. 
To facilitate computation, the following way is convenient. 

Recall that every character $\chi \bmod{q}$ 
can be expressed as a product of characters $\chi_j \bmod{p_j^{a_j}}$, $1\le j\le h$.
Assume, at first, that all the $p_j$'s are odd, or if $p_j=2$
for some $j$ then $a_j < 3$. Then the theory of primitive roots applies,
and we require $\chi$ to be presented to the algorithm as an $h$-tuple
of roots of unity $(\omega_1,\ldots,\omega_h)$ whose entries satisfy
$\omega_j^{m_j}=1$, where $m_j:=\phi(p_j^{a_j}) = p_j^{a_j-1}(p_j-1)$,
together with an $h$-tuple of primitive roots $(g_1\bmod{p_1^{a_1}},\ldots,g_h\bmod{p_h^{a_h}})$.
Given such tuples, the algorithm 
defines $\chi_j(g_j) := \omega_j$, for $1\le j\le h$, which determines $\chi$ uniquely.
If $p_j=2$ and $a_j \ge 3$ for some $j$ (so $\chi_j$ is a 
character $\textrm{mod}\,\,2^{a_j}$), then the entries corresponding to $p_j$  
in the above tuples are omitted, and we require an additional 2-tuple
$(\omega_1',\omega_2')$  whose entries satisfy $(\omega_1')^2=1$ and $(\omega_2')^{2^{a-2}}=1$.
The reason for this modification 
is that, taking $a=a_j$, the group $\left(\mathbb{Z}/2^a\mathbb{Z}\right)^{*}$ is 
not cyclic if $a\ge 3$, and so there is no primitive root. Therefore, we rely on the
well-known group decomposition of $\left(\mathbb{Z}/2^a\mathbb{Z}\right)^{*}$ 
to express the odd residue classes in the form $(-1)^{v_1} 5^{v_2}\bmod{2^a}$, 
where $v_1$ and $v_2$ are integers that are 
uniquely determined modulo $2$ and modulo $2^{a-2}$, respectively.
Last, given the 2-tuple $(\omega_1',\omega_2')$, the algorithms define $\chi_j(-1) := \omega_1'$ and 
$\chi_j(5) := \omega_2'$.

With $\chi$ thus presented, 
we supply a fast procedure for computing $\chi$ at individual points. 
Note that, in general, the problem of determining the value of $\chi \bmod{q}$ at an individual point 
is a hard discrete log problem. There are known sub-exponential time algorithms for solving it but their 
running times are only conjectural: see~\cite{O1} for a survey of such 
  algorithms. Fortunately, in our case, we can exploit
  the power-full structure of the modulus, which makes 
  the problem computationally simple.
\begin{lem}\label{lem:compC}
There are absolute constants $\kappa_4$, $\kappa_5$, and $\kappa_6$ such
that for any positive integer $q=p_1^{a_1}\cdots p_h^{a_h}$ (where $p_j$ are distinct primes), 
any given Dirichlet character $\chi\bmod q$, 
 any positive $\epsilon < e^{-1}$, and any integer $0\le c < q$,
the value of $\chi(c)$ can be computed to within
$\pm \epsilon$ using $O(\log^{\kappa_4}(q/\epsilon))$ operations on
numbers of $O(\log (q/\epsilon))$ bits, provided a precomputation, that depends on $q$ only, costing
$O((p_1+\cdots +p_h)\, \log^{\kappa_5} q)$ operations, and 
requiring $O((p_1+\cdots +p_h)\,\log^{\kappa_6} q)$ bits of storage, is performed.
Big-$O$ constants are absolute.
\end{lem}
\begin{proof}
It suffices to show how to compute each character $\chi_j\bmod{p_j^{a_j}}$
occurring in the decomposition $\chi=\chi_1\cdots\chi_h$. This is because 
there are only $h \ll \log q$ such characters,
and so the cost of computing $\chi(c)$ is the same as $\chi_j(c)$ except 
for an additional multiplicative factor of $\log q$, 
which falls within the target complexity of the lemma. 
In turn, to compute $\chi_j(c)$, it suffices to solve the discrete 
log problem $g_j^x \equiv c \bmod{p_j^{a_j}}$, because then $\chi_j(c)$ 
can be computed via the formula $\chi_j(c)=\omega_j^x$, which is
fast since $\omega_j$ is supplied to the algorithm via the presentation of $\chi$. So the difficult part
of computing $\chi_j(c)$ is to solve for $x$, which we do next.
(If $p$ divides $c$
then $\chi_j(c) = 0$; as this condition can be checked quickly by a single division,
 we may assume that $\gcd(p_j,c)=1$ from now on.)

Let us first deal with the odd $p_j$ case. 
Recall that $g_j$ is the primitive root associated with $\chi_j \bmod{p_j^{a_j}}$, 
and is supplied to algorithm via the presentation of $\chi$.
To avoid notational clutter, let $p=p_j$, $a=a_j$, and $g=g_j$.
In order to solve $g^x \equiv c \bmod{p^a}$, it suffices to find integers 
$l_1$ and $l_2$ such that $(g^{p-1})^{l_1} \equiv c^{p-1} \bmod{p^a}$ and 
$(g^{p^{a-1}})^{l_2} \equiv c^{p^{a-1}}\bmod{p^a}$. This is 
because, given $l_1$ and $l_2$, one can find integers $r$ and $s$ 
via the Euclidean algorithm (which is computationally fast)  such that $r(p-1)+sp^{a-1}=1$, 
and so $x = r (p-1)l_1 + s p^{a-1}l_2$ is a solution. Therefore, the discrete
log problem $\textrm{mod}\,\,p^a$ can be reduced to two discrete log problems
in the (cyclic) subgroups of $\left(\mathbb{Z}/p^a\mathbb{Z}\right)^{*}$ 
of order $p^{a-1}$ and $p-1$ (which are generated by $g^{p-1}$ and $g^{p^{a-1}}$, respectively).
Furthermore, the problem in the subgroup of order $p^{a-1}$
can be reduced to $a-1$ discrete log problems in the subgroup of order $p$ 
using a straightforward recursive procedure described~\cite{Pom} 
by Pomerance. For the convenience of the reader, let us
sketch that procedure here. We may assume $a\ge 2$, otherwise
the problem is either trivial or is already in the subgroup of order $p$.
We want to solve $(g^{p-1})^{l_1} \equiv c^{p-1}\bmod{p^a}$.
Since $l_1$ can be expressed in the form
$l_1=b_0+\cdots+b_{a-2}\,p^{a-2}$, $0\le b_r < p$, 
it suffices to determine the integers $b_r$
(which are the base-$p$ digits of $l_1$). 
To this end, suppose $b_0,\ldots,b_{r-1}$ are known, and 
let 
\begin{equation}
\begin{split}
\alpha_r &\equiv c^{p-1} (g^{p-1})^{-b_0-\cdots-b_{r-1}p^{r-1}} \bmod{p^a} \\ 
& \equiv g^{(p-1)(b_rp^r+\cdots+b_{a-2}p^{a-2})}\bmod{p^a}\,.
\end{split}
\end{equation}
Then, visibly, $\alpha_r$ is $(p-1)p^r$--power. So, letting $\beta_r \equiv \alpha_r^{p^{a-r-2}} \bmod{p^a}$,
we deduce that $\beta_r$ is a $(p-1)p^{a-2}$--power, and thus $\beta_r$ is in the subgroup of
order $p$, which is generated by $g'\equiv g^{(p-1)p^{a-2}} \bmod{p^a}$. Therefore, the solution of 
$(g')^x \equiv \beta_r \bmod{p^a}$, which is a discrete-log problem in the subgroup of order $p$, 
satisfies $x \equiv b_r\bmod{p}$, because, by definition, $g^{(p-1)p^{a-1}y}\equiv 1\bmod{p^a}$ for any $y\in\mathbb{Z}$,
and so
\begin{equation}
\begin{split}
\beta_r & \equiv \alpha_r^{p^{a-r-2}} \equiv g^{(p-1)(b_rp^r+\cdots+b_{a-2}p^{a-2}) p^{a-r-2}} \bmod{p^a}\\
&\equiv (g^{(p-1)p^{a-2}})^{b_r} \equiv (g')^{b_r} \bmod{p^a}\,,
\end{split}
\end{equation}
Moreover, $x$ determines $b_r$ uniquely since, by hypothesis, $0\le b_r < p$. 
As for $b_0$, which is needed to initialize the procedure,  
it is found simply by solving  $(g')^x \equiv c^{(p-1)p^{a-2}}\bmod{p^a}$, which is again  
a discrete log problem in the subgroup of order $p$ (to which $c^{(p-1)p^{a-2}}\bmod{p^a}$ belongs). We note that 
quantities like $\alpha_r$, $\beta_r$, $g'$, and $c^{(p-1)p^{a-2}}\bmod{p^a}$ can always be computed
using repeated squaring $\bmod{\,\,p^a}$, which is fast.
%otherwise the computation might take too long 
%and the numbers involved might contain more bits than allowed by the lemma.

In summary, to find $l_1$  it suffices to work with the generator 
$g'$ of the subgroup of order $p$.
One then tabulates its powers: $g',\ldots,(g')^p \bmod{p^a}$, by sequentially
 multiplying by $g'$ modulo $p^a$. It is important that this precomputation does not 
 depend on $c$, but only on $p^a$ (it is even independent of the character), 
 and so the table need not be created anew for
 each different $c$.
The cost of creating the table is about $p$ operations and $p$ space
(up to a poly-log factor in $p^a$). Once done, the value of $l_1$  
can be determined from the precomputed values using $a-1$ repetitions of  
procedure we have described, where each repetition involves a table look-up, 
which can be done in poly-log time in $p^a$, and hence in $q$ 
(assuming a random access memory model).

As for $l_2$, it suffices to work with the generator 
$g'' \equiv g^{p^{a-1}}\bmod p^a$ of the subgroup of order $p-1$.
As before, one tabulates its powers $g'',\ldots,(g'')^{p-1} \bmod{p^a}$,
so then $l_2$ can be determined by a direct table look-up. The overall cost
of this is again about $p$ time and $p$ space.

It remains to show how to compute individual values of 
$\chi_j \bmod{2^a}$, where $a\ge 3$. 
(The cases $a\in\{1,2\}$ do not represent any computational difficulty.) 
The main task here is to solve for $v_1$ and $v_2$ such that 
$(-1)^{v_1} 5^{v_2} \equiv c \bmod{2^a}$.
The index $v_1$ is simple to compute: it is either 
0 or 1 according to whether $c$ is
1 or $-1$ modulo 4. As for $v_2$, it can be computed via
a recursive procedure similar to the case of odd $p$
(one works in the cyclic subgroup generated by 5, which has order $2^{a-2}$).
Last, $\chi_j(c) = (\omega_1')^{v_1} (\omega_2')^{v_2}$.
\end{proof}

The next needed ingredient is the Postnikov character formula,
which was derived by Postnikov to obtain upper bounds on character sums. 
It was later re-proved by Gallagher~\cite{Ga}.
(See \cite{I} and \cite{IK} for other formulations.)  
The formula shows that the values of a Dirichlet character $\tilde{\chi} \bmod{p^a}$ 
along the arithmetic progression \mbox{$\{1,1+p^b,1+2p^b,1+3p^b,\ldots\}$}, with
$b\ge 1$,
are of the form $\tilde{\chi}(1+p^bx)=\exp(2\pi i f(x))$, where $f(x)$ is a polynomial 
with rational coefficients that depends on $\tilde{\chi}$, $p$, $a$, and $b$ only. 
Lemma~\ref{lem:l1} below, which is a specialization of lemma~2 in \cite{Ga},
 shows that $b$ can be arranged so that  $f(x)$ is of degree at most 2. 
\begin{lem}\label{lem:l1}
Let $\tilde{\chi}\bmod{p^a}$ be a Dirichlet character,
and let $b:=\lceil a/3\rceil$. 
If $p$ is an odd prime, then there exists an integer $L$, depending on
$\tilde{\chi}$, $p$, $a$, and $b$ only (so independent of $x$),  such that
\begin{equation}
\tilde{\chi}(1+p^b x)=\exp\left(\frac{4\pi i \,L\, x}{p^{a-b}}- \frac{2\pi i \,L\, x^2}{p^{a-2b}}\right)\,, 
\end{equation}
for all $x\in \mathbb{Z}$. If $p=2$ and $a > 3$, then there exists an integer $L_1$, depending on 
$\tilde{\chi}$ and $b$ only (so independent of $x$), such that
\begin{displaymath}
\tilde{\chi}(1+2^bx) = \exp\left(\frac{2\pi i \,L_1\, x}{2^{a-b}}- \frac{\pi i \,L_1\, x^2}{2^{a-2b}}\right)\,,
\end{displaymath}
for all $x\in\mathbb{Z}$. And if $p=2$ and $a\le 3$, then there exist
absolute constants $-1 \le L_2,L_3 \le 2$ such that  
$\tilde{\chi}(1+2^bx) = (-1)^{(L_2 x + L_3 x^2)/2}$ for all $x \in\mathbb{Z}$. 
\end{lem}
\noindent
\textit{Remark}. We distinguish the conclusion of the lemma 
for odd $p$ when  $a\in\{1,2,4\}$, where we have 
$\tilde{\chi}(1+p^b k)=e^{2\pi i L k/p^{a-b}}$ if $a\in\{2,4\}$, and 
$\tilde{\chi}(1+p^b k)=1$ if $a=1$, which is trivial.
\begin{proof}
Let $H$ be the kernel of the reduction homomorphism 
$\left(\mathbb{Z}/p^a\mathbb{Z}\right)^* \to \left(\mathbb{Z}/p^b\mathbb{Z}\right)^*$.
So, $H$ is a subgroup in $\left(\mathbb{Z}/p^a\mathbb{Z}\right)^*$
 consisting of the residue classes congruent to $1\bmod{p^b}$, and 
$H$ has size $|H|=p^{a-b}$.
Using our model for $\mathbb{Z}/p^a\mathbb{Z}$, 
the elements of $H$ are identified with the set of integers 
$\{1+p^b x\, |\, 0\le x <p^{a-b}\}$.

Assuming $p>2$, we construct a character 
$\psi$ of $H$ which generates the full character 
group of $H$, including the character $\left.\tilde{\chi}\right|_H$,
such that $\psi$ is given explicitly by a quadratic exponential. 
To this end, define the polynomial $f(x):=2x-x^2$. Then, for all $x,y\in \mathbb{Z}$, we have
\begin{equation} \label{eq:re1}
\begin{split}
f(p^b x+p^b y+p^{2b} x \,y) &= 2p^b x+2p^b y+2p^{2b} x\,y-p^{2b} x^2-p^{2b} y^2 \\
&\quad -2p^{2b} x\, y-2 p^{3b} x^2\,y -2p^{3b} x\,y^2-p^{4b} x^2 y^2 \\ 
&\equiv f(p^b x)+f(p^b y) \bmod{p^a}\,,
\end{split}
\end{equation}
where we made use of the relation $p^{3b} \equiv 0 \bmod{p^a}$, 
which holds due to our choice $b=\lceil a/3\rceil$. 
Consider the following function $\psi: H\to \mathbb{C}$ defined by
\begin{eqnarray}
\psi(1+p^b x):=\exp\left(\frac{2\pi i f(p^b x)}{p^a}\right)=\exp\left(\frac{4\pi i\, x}{p^{a-b}}-\frac{2\pi i\, x^2}{p^{a-2b}}\right)\,.\nonumber
\end{eqnarray}
Notice this definition is independent of the model for $\mathbb{Z}/p^a\mathbb{Z}$ as it 
yields the same result if $1+p^bx$ is replaced by $1+p^bx + p^ak$ for all $k\in \mathbb{Z}$.
Now, by the congruence relation (\ref{eq:re1}), we have
\begin{eqnarray} \label{eq:re2}
\psi((1+p^b x)\,(1+p^b y))=\psi(1+p^b x)\,\psi(1+p^b y)\,,
\end{eqnarray}
and this equality holds for all $x,y,\in\mathbb{Z}$.
Therefore, $\psi$ is multiplicative ($\psi$ respects the 
group operation in $H$). Moreover, $\psi$ is not identically zero, because $\psi(1) = 1$. 
Hence, $\psi$  must be a character of $H$. By a direct calculation, 
$f(p^b) \equiv 0 \bmod{p^b}$, or, equivalently, $f(p^b)/p^b$ is an integer. 
Moreover, since $b>0$ and $p>2$, we have $f(p^b) = 2p^b - p^{2b} \not\equiv 0 \bmod{p^{b+1}}$,
and so $f(p^b)/p^b$ is relatively prime to $p$.
This immediately implies that the values $\psi(1+p^b)^u = e^{2\pi i u (f(p^b)/p^b)/p^{a-b}}$, 
$0\le u < p^{a-b}$, are all distinct. 
In particular, $\psi$ has order $p^{a-b}$, which is the same as the order of $H$. 
Therefore, the powers of $\psi$ span the full character group of $H$. 

But $\left.\tilde{\chi}\right|_H$ is a character of $H$. 
Hence, there is an integer $L$ 
such that $\left. \tilde{\chi}\right|_{H} \equiv \psi^L$. To find $L$,
note that $\tilde{\chi}(1+p^b)=\exp(2\pi i B /p^{a-b})$ 
for some integer $B$ depending on $\tilde{\chi}$, $p$, and $b$ only.
(In our application, $B$ can be determined quickly using lemma~\ref{lem:compC}.) So,
$L$ can be computed by simply solving the congruence $L\,f(p^b)/p^b \equiv B \bmod{p^{a-b}}$, which yields 
\begin{equation}\label{eq:Lcong}
L \equiv B\,\overline{2-p^b}\,\bmod{p^{a-b}}\,.
\end{equation}

It remains to consider the case when the modulus is $2^a$.
If $a> 3$, then the same derivation as in the odd prime case 
applies except one uses the polynomial $f_1(x) = x - x^2/2$
instead of $f(x)$, which gives
\begin{equation}\label{eq:Lcong1}
L_1 \equiv B\,\overline{1-2^{b-1}}\,\bmod{2^{a-b}}\,.
\end{equation}
(Notice $f_1(x)$ consists of the first two terms in the Taylor expansion of $\log(1+x)$). 
If $a\le 3$, then the previous proof does not go through because $b=1$ and so 
the condition $f_1(p^b) \not\equiv 0 \bmod{p^{b+1}}$ fails.
Nevertheless, if $a\le 3$, then we can find integers $L_2$ and $L_3$ such that 
$\tilde{\chi}(1+2^bx) = (-1)^{(L_2 x + L_3 x^2)/2}$.
Specifically, if $\tilde{\chi}$ is the principal character,
which is the sole character if $a=1$, then take $L_2=L_3=0$. 
So, we may assume $\tilde{\chi}$ is not principal. 
If $a=2$, then there is a single non-principal character, for which we take $L_2=2,\, L_3=0$. And
if $a=3$, then take $L_2=1,\,L_3=-1$ or $L_2=2,\,L_3=0$, 
or $L_2=1,\,L_3=1$, according to whether $\tilde{\chi}(3)=1$ and $\tilde{\chi}(5)=-1$, 
or $\tilde{\chi}(3)=-1$ and $\tilde{\chi}(5)=1$, or $\tilde{\chi}(3)=-1$ and $\tilde{\chi}(5)=-1$,
respectively. The validity of these choices of $L_2$ and $L_3$ can be verified
 by inspection.
\end{proof} 

The last needed ingredient is the following
algorithm for computing quadratic exponential sums
(Theorem 1.1 in \cite{H1}). 
\begin{thm}\label{thetaalg}
There are absolute constants $A_6$, $A_7$, $A_8$, $\kappa_7$, and $\kappa_8$ 
such that for any positive \mbox{$\epsilon<e^{-1}$}, any integer $K>0$, any integer $j \ge 0$, any $\alpha,\beta \in [0,1)$, and with  \mbox{$\nu:=\nu(K,j,\epsilon)=(j+1)\log (K/\epsilon)$},  the value of the function 
\begin{displaymath}
\frac{1}{K^j}\sum_{0\le k <K} k^j\, e^{2\pi i \alpha k+2\pi i \beta k^2}\,,
\end{displaymath}
can be computed to within $\pm \, A_6\, \nu ^{\kappa_7} \epsilon$ using \mbox{$\le A_7\, \nu^{\kappa_8}$} arithmetic operations on numbers of $\le A_8\, \nu^2$ bits.
\end{thm}

By combining lemma~\ref{lem:compC}, lemma~\ref{lem:l1}, and Theorem~\ref{thetaalg}, we obtain the
following algorithm for computing theta sums twisted by a character $\chi$. This algorithm is  
how the power-savings in computing $L(s,\chi)$ will be achieved in Theorem~\ref{mainthm} later.
\begin{thm}\label{lem:l2}
There are absolute constants $A_9,\ldots,A_{12}$, $\kappa_9,\ldots,\kappa_{11}$,
such that for any positive integer $q = p_1^{a_1}\cdots p_h^{a_h}$ (where $p_j$ are distinct primes),
any given character $\chi\bmod q$, any positive $\epsilon < e^{-1}$, any integer $K>0$, any integer $v$, and any integer $j\ge 0$, 
and with  $\nu_1:= \nu_1(K,q,v,\epsilon) = (j+1)\log (qK(|v|+1)/\epsilon)$, the function
\begin{equation}
S_{\chi}(K,v,j;\alpha,\beta):=\frac{1}{K^j}\,\sum_{0\le k < K} k^j\, \chi (v+k)\,e^{2\pi i \alpha k+2\pi i \beta k^2}\,,
\end{equation} 
can be computed to within $\pm \epsilon$ using 
$\le A_9 \, p_1^{\lceil a_1/3\rceil} \cdots p_h^{\lceil a_h/3\rceil}\,\nu_1^{\kappa_9}$ 
operations on numbers of $\le A_{10}\,\nu_1^2$ bits,
provided a precomputation, that depends on $q$ only, costing
$\le A_{11} \, (p_1+\cdots+ p_h) \,\log^{\kappa_{10}} q$ operations, and requiring 
$\le A_{12} \, (p_1+\cdots +p_h) \,\log^{\kappa_{11}} q$ bits of storage, is performed. 
\end{thm}
\noindent
\textit{Remark}. The precomputation requirement comes directly from lemma~\ref{lem:compC}. 
The constants $\kappa_{10}$ and
$\kappa_{11}$ are the same as $\kappa_5$ and $\kappa_6$ in lemma~\ref{lem:compC}, respectively.
\begin{proof}
Since $\chi$ has period $q$, we have $\chi(n) = \chi(\tilde{n})$, where 
$\tilde{n} := n-\lfloor n/q\rfloor$. As $\tilde{n}$ can be computed in poly-log time (in $q$ and $n$), 
and as $0\le \tilde{n} < q$, then we only need to know how to compute $\chi(n)$ for $0\le n < q$.
By lemma~\ref{lem:compC}, once a precomputation costing 
$O((p_1+\cdots +p_h) \,(\log q)^{\kappa_5})$ operations and requiring
$O((p_1+\cdots +p_h) \,(\log q)^{\kappa_6})$
 bits of storage is performed,  the value of $\chi(n)$ for any $0\le n<q$ can be computed to within 
 $\pm \epsilon/(2K)$ using $O(\log^{\kappa_4}(qK/\epsilon))$ operations on numbers of 
 $O(\log (qK/\epsilon))$ bits using the precomputed values. Since such a precomputation 
 is permitted by the theorem, we may assume from now on that $\chi(n)$ can be 
 computed to within $\pm \epsilon/(2K)$ for any $0\le n<K+v$ in $\ll \nu_1^{\kappa_4} +\log(|v|+1) \ll \nu_1^{\kappa_4+1}$ time.

Let us first prove the lemma in the simpler situation 
$v,j,\alpha,\beta=0$; i.e. for $S_{\chi}(K)$. To this end, 
define $C:= C_q=p_1^{\lceil a_1/3\rceil} \cdots p_h^{\lceil a_h/3\rceil}$ and $K_l := K_{l,C}= \lceil (K-l)/C \rceil$.
Then, split the range of summation in $S_{\chi}(K)$ into arithmetic progressions 
\begin{equation}\label{eq:qqq}
S_{\chi}(K) = \sum_{\substack{0\le l < C\\ (l,q)=1}} \chi(l)\sum_{0\le k < K_l} \chi(1+\overline{l}\, C\,k)\,.
\end{equation}
Now, $\chi=\chi_1\cdots\chi_h$, where $\chi_j \bmod{p_j^{a_j}}$.
So $\chi(1+\overline{l}Ck) = \chi_1(1+\overline{l}Ck)\cdots \chi_h(1+\overline{l}Ck)$.
Applying lemma~\ref{lem:l1} to each $\chi_j(1+\overline{l}Ck)$ separately, with $\tilde{\chi}=\chi_j$, 
$a = a_j$, $b = \lceil a_j/3\rceil$,
and $x = \overline{l}C k / p^b$ (note that $x$ is an integer since, by definition, $p^b$ divides $C$), 
we can express each $\chi_j(1+\overline{l}Ck)$ as a quadratic exponential in $k$.
Each such application of lemma~\ref{lem:l1} involves two steps. 
First, one determines the integer $0\le  B < p^{a-b}$ satisfying 
$\tilde{\chi}(1+p^b)=\exp(2\pi i B/p^{a-b})$, which is straightforward since 
$\tilde{\chi}(1+p^b)$ can be computed 
using the already precomputed look-up tables from lemma~\ref{lem:compC}. 
Second, one solves the congruence \eqref{eq:Lcong} or \eqref{eq:Lcong1} 
for $L$, which can be done fast via the Euclidean algorithm.
Put together, the inner sum in (\ref{eq:qqq}) can be expressed in the form
\begin{equation}\label{eq:qq1}
\begin{split}
\sum_{0\le k < K_l} \chi(1+\overline{l} \,C\,k) &= \sum_{0\le k < K_l} \chi_1(1+\overline{l} \,C\,k) \cdots \chi_h(1+\overline{l} \,C\,k) \\
&= \sum_{0\le k < K_l} e^{2\pi i \alpha_1 k+2\pi i \beta_1 k^2}\,,
\end{split}
\end{equation} 
where $\alpha_1,\beta_1\in [0,1)$ are constants, depending on $\chi$, $C$, and $\overline{l}$ only 
(so independent of $k$), whose values 
can be determined quickly by solving at most $h$ congruences like \eqref{eq:Lcong}
and \eqref{eq:Lcong1}. By Theorem~\ref{thetaalg}, the exponential sum on the r.h.s. of (\ref{eq:qq1}) 
can each be computed to within $\pm \epsilon/(2C)$ in poly-log time (in $K$ and $C/\epsilon$). Since there are
at most $C$ such sums to be computed, the lemma follows
for $S_{\chi}(K)$.

We extend the previous method to the generalized sum $S_{\chi}(K,v,j;\alpha,\beta)$.
To begin, define the coefficients $d_{l,r}:=d_{l,r,j,C,K}$ via the binomial expansion
\begin{equation}
K^{-j} (l+Ck)^j =: \sum_{r\le j} d_{l,r} (k/K_l)^r\,.
\end{equation}
So $d_{l,r}$ are explicitly given by
\begin{eqnarray}\label{eq:dlr}
d_{l,r}=\binom{j}{r} \,\frac{l^{j-r}\,C^r\,(K_l)^r}{K^j}\,.
\end{eqnarray}
Then,  lemma~\ref{lem:l1} yields
\begin{equation}\label{eq:qqq1}
\begin{split}
&S_{\chi}(K,v,j;\alpha,\beta)  \\
&\quad = \sum_{\substack{0\le l < C\\ (l+v,q)=1}} \chi(l+v)\,e^{2\pi i \alpha l + 2\pi i \beta l^2}\,\sum_{0\le k < K_l} K^{-j}(l+C\,k)^j\,\chi(1+\overline{l+v}\, C\,k) e^{2\pi i (\alpha+2l\beta) C k + 2\pi i \beta C^2 k^2}\\
&\quad = \sum_{\substack{0\le l < C\\ (l+v,q)=1}} \chi(l+v)\,e^{2\pi i \alpha l + 2\pi i \beta l^2}\,\sum_{0\le r \le j} \frac{d_{l,r}}{(K_l)^r}\sum_{0 \le k < K_l} k^r e^{2\pi i ((\alpha+2l\beta)C+\alpha_2) k+2\pi i (\beta C^2+\beta_2) k^2}\,,
\end{split}
\end{equation}
where $\alpha_2,\beta_2\in [0,1)$ are constants, depending on $\chi$, $C$, and $\overline{l+v}$  
only (so independent of $k$), whose values can be 
computed quickly by solving $\le h$ congruences like \eqref{eq:Lcong}
and \eqref{eq:Lcong1}. 

We digress briefly to discuss how to compute $d_{l,r}$.
There are several ways for doing this; the  following 
suffices for the current exposition. For each $r$, 
one precomputes the factorials $r!$,
$(j-r)!$, and $j!$, exactly,
by sequential multiplication of integers.
Since $r\le j$, this can be done using $\le j$ operations on integers of
$\ll (j+1)\log(j+1)\ll (j+1)^2$ bits\footnote{This step (and similar ones involving binomial coefficients) 
requiring $(j+1)\log(j+1)\ll (j+1)^2 \ll \nu_1^2$ bits, 
is essentially the reason why this theorem is stated with the upper bounds $O(\nu_1^2)$ on the number of bits
(and Theorem~\ref{thetaalg} (\cite[Theorem 1.1]{H1}) was stated with the upper bound $O(\nu^2)$ on the number of bits).
Otherwise, all that is required is $\ll j + \log(qK/\epsilon)$-bit arithmetic (and 
$\ll j+\log(K/\epsilon)$-bit arithmetic, respectively). It is plain 
that one can prove this is in fact all is required.}
and requiring $\ll  (j+1)^2$ bits of storage, which is allowed by the theorem. 
The binomial coefficient can then be computed to
within $\pm \epsilon/(4C(j+1))$ in the form $j!/(r!(j-r)!)$ 
which requires three operations on numbers of $\ll (j+1)^2$ bits using the precomputed values of the factorial.
Also, each of $l^{j-r}$, $C^r$, $(K_l)^r$, and $K^j$, can be computed exactly using
$\ll j$ operations on numbers of $O((j+1)\log(qK))$ bits.
%In practice, however, one can pay attention to the order of operations when computing $d_{l,r}$ in order to 
%minimize the number of required bits. Consequently, it useful to
%consider the possibilities $K_l \ge j$ and $K_l < j$ separately.
%If $K_l \ge j$, which is typically the case in our application, 
%then by the facts $\binom{j}{r}\le j^{j-r}$, $l<C$, 
%and $K \le CK_l$,  we have $d_{l,r} < (j/K_l)^{j-r} \le 1$.
%So $d_{l,r}$ are typically of bounded size. 
%If $K_l<j$, then one can compute the sum directly at a cost 
%of $\ll K_l\nu_1 \ll  \nu_1^2$ operations, which falls within the target complexity.
%So, without loss of generality, we may assume that $j \le K$.

To conclude, then, define $\tilde{\nu}:=\nu(K,j,\epsilon/q) = (j+1)\log(qK/\epsilon)$.
Theorem~\ref{thetaalg} ensures that each quadratic  sum in (\ref{eq:qqq1}) (the inner-most sums
in the last line) can be computed to 
within $\pm \epsilon/(4C(j+1))$ using $O(\tilde{\nu}^{\kappa_8})$ operations on numbers of 
$O(\tilde{\nu}^2)$ bits (since, by assumption, $j\le K$). Since there are $\le C(j+1)$ such sums, and since
the precomputed look-up tables required by lemma~\ref{lem:compC} are already available,
so each $\chi(l+v)$ can be computed to within $\pm \epsilon/(4C(j+1))$ 
using $O(\nu_1^{\kappa_4+1})$ operations,
then, on noting $\tilde{\nu} \le \nu_1$, we see that the overall
cost of computing $S_{\chi}(K,v,j;\alpha,\beta)$ to within $\pm \epsilon$ is 
$O(C (\nu_1 + \nu_1^{\kappa_4+1} +\nu_1^{\kappa_8+1}))$ operations. The theorem follows.
\end{proof}

\section{Application: computing $L(s,\chi)$} \label{sec: app}

We would like the starting point in this section to be an unsmoothed approximate functional equation
for $L(s,\chi)$ (i.e. a ``Riemann-Siegel'' type formula). This is because unsmoothed formulae make
it far simpler to apply subdivisions to the main sum, as we will do. 
On the downside, unsmoothed formulae of length square-root of the analytic conductor
are quite complicated to derive.
One might appeal to the main theorem in~\cite{Da}, for example,
which provides an unsmoothed formula, but which does not apply when $s$ is small,
and the explicit asymptotic constants in its remainder term have not been worked out explicitly.
Fortunately, given Theorem~\ref{lem:l2}, we can circumvent these difficulties easily, at least for
the purpose of the theoretical derivation. The reason is that, Theorem~\ref{lem:l2} will yield
the same power-saving for computing $L(s,\chi)$ even if one starts with a main sum of 
length $\lceil q^d(|s|+1)^d \rceil$, where $d$ is any fixed number, because it 
will be applied locally, to blocks in the main sum, and it depends on the 
block-length and the required precision in a poly-log way only.  

To this end, let $\chi\bmod{q}$ be a non-principal character, where $q=p_1^{a_1}\cdots p_h^{a_h}$. 
We first consider the case when 
$\chi$ is primitive. As before, let $\mathfrak{a} := (1-\chi(-1))/2$,
so $\mathfrak{a}$ is 0 or 1 according to whether $\chi$ is even or odd. Define
\begin{equation}
\xi(s,\chi):= \left(\frac{q}{\pi}\right)^{\frac{s}{2}} \Gamma\left(\frac{s+\mathfrak{a}}{2}\right) L(s,\chi)\,,
\end{equation}
and $\overline{\xi}(s,\chi):=\overline{\xi(\overline{s},\chi)}=\xi(s,\overline{\chi})$. 
We have the following functional equation 
\begin{equation}\label{eq:xifunceq}
\overline{\xi}(1-s,\chi) = \frac{i^{\mathfrak{a}}q^{1/2}}{\tau(\chi)} \xi(s,\chi)\,.
\end{equation}
Therefore, we may restrict our
computations of $L(s,\chi)$ to the half-plane $\Re(s)\ge 1/2$
since values of $L(s,\chi)$ elsewhere can be recovered routinely 
by the functional equation.
Here, the Gauss sum $\tau(\chi)$ can be computed to within $\pm \epsilon$ by Theorem~\ref{lem:l2},
which consumes about $p_1^{\lceil a_1/3\rceil}\cdots p_h^{\lceil a_h/3 \rceil}$ time,
up to a poly-log factor (in $q$ and $1/\epsilon$).
While the functional equation is valid for primitive $\chi$ only,
it will be apparent that our use of it is not essential, provided we restrict the 
computations to the half plane $\Re(s)>\sigma_0$, where $\sigma_0>0$ is fixed.
Alternatively, one can use the
expression for $L(s,\chi)$ in terms of $L(s,\chi_1)$, where $\chi_1$ is the 
primitive character inducing $\chi$, to enable 
the functional equation for $L(s,\chi_1)$ to be used instead.
However, for simplicity, we will assume $1/2\le \Re(s) \le 1$, say, from now on.

We use the P\'olya-Vinogradov inequality
to reduce the computation of $L(s,\chi)$ to computing 
a main sum $\sum_{n<M} \chi(n) n^{-s}$, where $M$ is chosen according to the desired 
precision. Specifically, if $\chi\bmod{q}$ is primitive then 
$|\sum_{N_1\le n<N_2} \chi(n)| < q^{1/2}\log q$,
and if $\chi \bmod{q}$ is induced by the primitive character $\chi_1\bmod{q_1}$, then  on
combining the estimates (see \cite[Chap. 23]{D})  
$|\sum_{N_1\le n<N_2} \chi(n)| < 2^{\omega(q/q_1)} (q_1)^{1/2}\log(q_1)$ and $2^{\omega(q/q_1)} \le d(q/q_1) \le 2 (q/q_1)^{1/2}$,
where $\omega(r)$ is the number of distinct prime factors of $r$, we obtain
\begin{equation}
|\sum_{N_1\le n<N_2} \chi(n)| < 2 (q/q_1)^{1/2} (q_1)^{1/2} \log(q_1) \le 2 q^{1/2} \log q\,.
\end{equation}
Therefore, by applying partial summation to $\sum_{n\ge M} \chi(n) n^{-s}$, we arrive at 
\begin{equation}\label{eq:lchi}
L(s,\chi) = \sum_{1\le n<M} \frac{\chi(n)}{n^s} + \mathcal{R}\,,
\end{equation}
where 
\begin{equation}\label{eq:lchierr}
|\mathcal{R}| \le \frac{2q^{1/2}\log q}{\Re(s) M^{\Re(s)}}\, (|s|+1)\,.
\end{equation}
For example, if $M \ge (6 q\log q)^{1/\Re(s)}$, 
then the main sum in \eqref{eq:lchi} approximates
$L(1/2,\chi)$ to within $\pm q^{-1/2}$.
As another example, if $q>900$ say, then 
we can ensure  $|\mathcal{R}|\le q^{-\lambda}(|s|+1)^{-\lambda}$  
by taking $M \ge q^d(|s|+1)^d $, $d = (\lambda+1)/\Re(s)$. 
Notice that, if the desired bound on $\mathcal{R}$ is reduced 
by a multiplicative factor of $\epsilon_1$, then $d$ changes to 
$d - \log(\epsilon_1)/(\Re(s) \log(q(|s|+1)))$, and so $d$ grows very slowly
(logarithmically) as we tighten the bound on $\mathcal{R}$.

We conclude that to prove Theorem~\ref{mainthm} it suffices to compute the main 
sum in \eqref{eq:lchi} with a suitable $M$.
Before presenting the proof, 
we take a brief detour to emphasize the following. While
the proof will yield the asymptotic complexity claimed 
in the theorem even when the length of the main sum is $M\ge
q^{100}(|s|+1)^{100}$, say,
it is better in practice to start with a shorter main sum such as the one provided
by a ``Riemann-Siegel'' type formula for $L(s,\chi)$ with explicit asymptotic 
constants in the remainder term.
On the other hand, if the available implementation of the algorithm 
of Theorem~\ref{lem:l2} is well-optimized, and supposing, for instance, that one wishes to compute 
$L(1/2,\chi)$, or perhaps only the low-lying zeros of $L(1/2+it,\chi)$, 
to within $\pm q^{-1/2}$ (so that choosing $M \ge (q(|s|+1))^3$ suffices),
then this issue might have a relatively little impact on the overall running time because, 
as mentioned before, the dependence of Theorem~\ref{lem:l2} on the length of the block and the
desired precision is only poly-log.
Therefore, in the proof of Theorem~\ref{mainthm},  
we prefer to retain the simplicity and uniformity provided by \eqref{eq:lchi},
as well as its indifference to whether the character is primitive or not.

\begin{proof}[Proof of Theorem~\ref{mainthm}]
Our goal is to prove an upper bound on the number of operations required
to compute $L(s,\chi)$ to within $\pm q^{-\lambda}(|s|+1)^{-\lambda}$,
where $\chi$ is a character $\textrm{mod}\,\,q$, and $q$ has prime
factorization $q=p_1^{a_1}\cdots p_h^{a_h}$.
The character $\chi$ should be presented to the algorithm as we detailed in \textsection{\ref{sec:compS}}.
Notice that the presentation of $\chi$ includes the factorization of $q$.
We assume, for convenience, that $q(|s|+1)\ge 10^3$.

In \eqref{eq:lchi}, we choose $M=\lceil q^d(|s|+1)^d \rceil$ where $d=(\lambda+2)/\Re(s)$.
This choice of $M$ ensures, via estimate \eqref{eq:lchierr}, the assumption $q(|s|+1)\ge 10^3$, 
and the hypothesis $1/2 \le \Re(s) \le 1$ in the statement of the theorem, that 
$|\mathcal{R}| \le 0.1\, q^{-\lambda}(|s|+1)^{-\lambda}$.
Next, we use lemma~\ref{lem:compC}, and the periodicity of $\chi$, to enable the 
evaluation of $\chi(n)$ to within $\pm 0.1\, q^{-\lambda}(|s|+1)^{-\lambda}/ M$ 
for any $0\le n< M$ using $O((\lambda+d+1)^{\kappa_4}\log^{\kappa_4}(q(|s|+1)))$ 
operations on numbers of $O((\lambda+d+1)\log(q(|s|+1)))$
bits. The lemma requires precomputing look-up tables, which is done only once throughout this proof. 
The precomputation costs $O((p_1+\cdots +p_h)\, \log^{\kappa_5} q)$ operations, and 
requires $O((p_1+\cdots+ p_h)\,\log^{\kappa_6} q)$ bits of storage, which
is permitted by the theorem.

Next, let $M_1 = M_{1,q,s} := p_1^{\lceil a_1/3\rceil}\cdots p_h^{\lceil a_h/3\rceil}\,\lceil (1+|s|)^{1/3}\rceil$,
and divide the main sum into an initial sum and a ``bulk sum''
\begin{equation}
\sum_{1\le n < M} \frac{\chi(n)}{n^s}=\sum_{1\le n < M_1} \frac{\chi(n)}{n^s}+\sum_{M_1 \le n< M} \frac{\chi(n)}{n^s}\,.
\end{equation}
By appealing to lemma~\ref{lem:compC} to compute individual values of $\chi$, 
we see that the initial sum can be 
evaluated directly, to within $\pm 0.1\, q^{-\lambda}(|s|+1)^{-\lambda}$,
using $O(M_1\,(\lambda+d+1)\, \log(q(|s|+1)))$ operations, which falls within our target complexity.
So it remains to deal with the ``bulk sum'', which is where 
the power-savings will be achieved.  We perform the following dyadic subdivision
\footnote{The following subdivision scheme is more efficient in practice than a dyadic subdivision (by a constant factor) 
because it yields larger blocks to feed into Theorem~\ref{lem:l2} later: Let $\tilde{v}_0 = M_1$, 
and sequentially define $\tilde{K}_r := \min\{\lceil \tilde{v}_r/(|s|+1)^{1/3}\rceil, M-\tilde{v}_r\}$,
$\tilde{v}_{r+1} := \tilde{v}_r + \tilde{K}_r$, to obtain
\begin{equation}
\sum_{M_1 \le n< M} \chi(n) n^{-s} = \sum_{0\le r < \tilde{R} } \sum_{0\le
k<\tilde{K}_r} \chi(\tilde{v}_r+k)(\tilde{v}_r+k)^{-s}\,, \nonumber
\end{equation}
where $\tilde{R}:=\tilde{R}_{s,M_1,M}$ can be shown to satisfy $\tilde{R} \ll (|s|+1)^{1/3}$.
The reason we use a dyadic subdivision in the proof, even though it is less efficient, is 
because it is likely more familiar, and so it might be marginally simpler.}
of the ``bulk sum''
\begin{equation}\label{eq:tailsum0}
\sum_{M_1 \le n< M} \frac{\chi(n)}{n^s} =\sum_{I\in\mathcal{I}} \sum_{n\in I} \frac{\chi(n)}{n^s} \,,
\end{equation} 
where $\mathcal{I}$ is the set of consecutive subintervals $I$
that partition $[M_1,M)$. Each subinterval in $\mathcal{I}$ is of the form
$I= [N,2N)$, $N \in [M_1,M)$, except possibly the last subinterval, which is of the form $[N,M)$.
In explicit terms, if we define $d_0 := \lfloor \log(M/M_1)/\log 2  \rfloor$, then 
$\mathcal{I} = \{[2^r M_1,2^{r+1}M_1), 0\le r < d_0\} \cup \{[2^{d_0}M_1, M )\}$. 
Note that 
\begin{equation}
|\mathcal{I}| \le \frac{\log(M/M_1)}{\log 2} + 1 \le 10 \, d\log (q(|s|+1)) \,.
\end{equation}
Therefore, if one plans on computing each inner sum in \eqref{eq:tailsum0} separately, 
as we will do, then computing the full sum will multiply the cost by an extra factor of $10 \, d\log (q(|s|+1))$ only,
which can be absorbed by our target complexity. Given this, it 
suffices to show how to compute each of the sums $\sum_{n\in I} \chi(n) n^{-s}$.

For each subinterval $I=[N,2N)$ (except possibly the last one,
which, in any case, is dealt with similarly), we define
$K := K_{N,s}= \lceil N/(|s|+1)^{1/3} \rceil$. 
We let $\mathcal{V}:=\mathcal{V}_{N,K}(I) = \{N,\ldots,N + \lfloor N/K\rfloor K\}$,
so $\mathcal{V}$ is a set of equidistant points in $[N,2N)$ separated by distance $K$.
Therefore, we have:
\begin{equation} \label{tailsum}
\begin{split}
\sum_{n\in I} \frac{\chi(n)}{n^s}&= \sum_{N\le n< 2N} \chi(n)\,e^{-s\log n}\\
&= \sum_{v\in\mathcal{V}} \sum_{0 \le k < K} \chi(v+k)\,e^{-s\log (v+k)} +\sum_{0\le k < K'} \chi(v'+k)\,e^{-s\log(v'+k)}\,,
\end{split}
\end{equation}
where the length of the second (tail) sum in \eqref{tailsum} satisfies $0\le K' <K$. 
Now, by definition, $N/K \le (|s|+1)^{1/3}$, and 
in particular $|\mathcal{V}|= \lfloor N/K \rfloor +1 \le (|s|+1)^{1/3}+1$.
So, to prove the theorem, it suffices to show that each inner sum in \eqref{tailsum} 
(as well as the tail sum, which is handled similarly) can be
computed in $p_1^{\lceil a_1/3\rceil}\cdots p_h^{\lceil a_h/3\rceil}$ times 
poly-log time. This will be accomplished via Theorem~\ref{lem:l2} as follows.
We apply the expansion $\log(1+x) = x - x^2/2 + x^3/3 -\cdots$ to $\log(1+k/v)$, to obtain
\begin{equation}\label{eq:afe4}
\begin{split}
\sum_{0 \le k < K} \chi(v+k)\,e^{-s\log (v+k)}&= e^{-s\log v} \sum_{0\le k < K} \chi(v+k)\,e^{-s\log (1+k/v)}\\
&= e^{-s\log v} \sum_{0\le k < K} \chi(v+k)\,e^{-s(\frac{k}{v} - \frac{k^2}{2v^2} + \frac{k^3}{3v^3}-\cdots)}\,.
\end{split}
\end{equation} 
By our choice of $K$, and the facts $v\ge N$ and $N\ge M_1$, it follows that 
$|k/v| \le (|s|+1)^{-1/3}$, and so the cubic and higher terms in
$s\log(1+k/v) = s k /v - s (k/v)^2/2 + s(k/v)^3/3 -\cdots$ are $O(1)$. 
More precisely, $|s|(k/v)^{3+j_1}/(3+j_1) \le (|s|+1)^{-j_1/3} \le (3/2)^{-j_1/3}$. 
Thus, using Taylor expansions (in the third equality below), we obtain 
\begin{equation}\label{lasteq}
\begin{split}
e^{-s(\frac{k}{v} - \frac{k^2}{2v^2} + \frac{k^3}{3v^3}-\cdots)} &= e^{-\frac{i\Im(s)}{v}k + \frac{i\Im(s)}{2v^2}k^2} e^{-\frac{\Re(s)}{v}k + \frac{\Re(s)}{2v^2}k^2 - \frac{s}{3v^3}k^3 +\cdots}\\
&= e^{-\frac{i\Im(s)}{v}k + \frac{i\Im(s)}{2v^2}k^2} \times \\
&\qquad e^{-\frac{\Re(s)K}{v}\frac{k}{K} + \frac{\Re(s)K^2}{2v^2}\frac{k^2}{K^2} - \frac{sK^3}{3v^3}\frac{k^3}{K^3} 
+\cdots \pm \frac{sK^{J_0}}{J_0 v^{J_0}}\frac{k^{J_0}}{K^{J_0}}} + \mathcal{E}^{'}_{s,v,k,J_0}\\
&=  e^{-\frac{i\Im(s)}{v}k + \frac{i\Im(s)}{2v^2}k^2} \sum_{0\le j < J} z_{j,s,v,J_0} \frac{k^j}{K^j} +
\mathcal{E}_{s,v,k,J} + \mathcal{E}^{'}_{s,v,k,J_0}\,.
\end{split}
\end{equation}
Since $|s|(k/v)^{3+j_1}/(3+j_1) \le (3/2)^{-j_1/3}$, 
the truncation error $\mathcal{E}^{'}_{s,v,k,J_0}$  satisfies 
$|\mathcal{E}^{'}_{s,v,k,J_0}| <  0.1 \, q^{-\lambda}(|s|+1)^{-\lambda}/(MK)$
when $J_0\ge \tilde{J}_0$, where $\tilde{J}_0 \ll (d+\lambda+1) \log (q(|s|+1))$. 
Thus, it suffices to choose $J_0 = \lceil\tilde{J}_0\rceil$.
We similarly claim that  $|\mathcal{E}_{s,v,k,J}| <  0.1 \, q^{-\lambda}(|s|+1)^{-\lambda}/(MK)$
when $J\ge \tilde{J}$, where $\tilde{J} \ll (d+\lambda+1) \log (q(|s|+1))$,  
and so it suffices to take $J = \lceil\tilde{J}\rceil$. 
To see why this bound on $\tilde{J}$ holds, consider the function $\eta(w) := e^{\tau_1 w + \cdots + \tau_{J_0} w^{J_0}}$,
where $\tau_1 := -\Re(s)K/v$, $\tau_2 := \Re(s)K^2/(2v^2)$, 
and $\tau_j :=  (-1)^{j+1} s K^j/(j v^j)$ for $3\le j\le J_0$, note that 
$\tau_1,\tau_2\ll 1$, and recall that 
$|\tau_{3+j_1}| \le (3/2)^{-j_1/3}$, then, by a standard application of Cauchy's
theorem, we obtain
$|z_{j,s,v,J_0}| \le (2\pi)^{-1} |\int_{|w|=5/4} \eta(w)/w^{j+1}\,dw| \ll (5/4)^{-j}$.
%We conclude that both $J_0$ and $J$ can be taken of poly-log size. 
Moreover, the coefficients $z_{j,s,v,J_0}$ (which are independent of $k$) can be computed fast as follows. 
Let $\tau_j$ be defined as before, and define the polynomials $P_r(w)$ via the recursion: 
$P_0(w) :=1$, $P_r(w) = (P'_{r-1}(w) + P_{r-1}(w) Q'(w))/r$, where $Q(w):= \tau_1 w + \cdots + \tau_{J_0} w^{J_0}$, 
and $P'(w)$ and $Q'(w)$ denote the derivative
with respect to $w$. Then it is fairly easy to see that $z_{j,s,v,J_0} = P_j(0)$. 
And to compute $z_{j,s,v,J_0}$, $0\le j \le J$,
it suffices to repeat the said recursion  $J+1$ times,  noting  that each repetition requires
$\ll (J+1)(J_0+1)$ operations only because it suffices to keep track of merely the first $J+1$ terms
in $P_r(w)$ throughout (we can discard the rest because $P_r(w)$ will be differentiated
at most $J$ times, then evaluated at zero). It follows that $z_{j,s,v,J_0}$, $0\le j \le J$, can 
be computed at a total cost $\ll (J+1)^2(J+1) \ll (d+\lambda+1)^3 \log^3 (q(|s|+1))$ operations.

By plugging \eqref{lasteq} back into \eqref{eq:afe4} 
and interchanging the order of summation, 
we see that the last sum in \eqref{eq:afe4} can be rewritten, 
to within $\pm\, 0.2 \, q^{-\lambda}(|s|+1)^{-\lambda} /M$,
a linear combination, with quickly computable coefficients, of $J+1$ sums of the form
\begin{equation}\label{eq:finsum}
\frac{1}{K^j} \,\sum_{0\le k < K} k^j\,\chi(v+k)\,e^{2\pi i \alpha k+2\pi i \beta k^2}\,,
\end{equation}
where $0\le j \le J \ll (d+\lambda+1) \log q(|s|+1)$, $\alpha = -\Im(s)/(2\pi v)$, and $\beta = \Im(s)/(4\pi v^2)$.
Letting $\nu_2:= (J+1)(d+\lambda+1) \log(q(|s|+1))\ll (\lambda+1)^2 \log^2(q(|s|+1))$,
it follows by Theorem~\ref{lem:l2} that 
each sum \eqref{eq:finsum} can be computed to within $\pm\, 0.1\, q^{-\lambda}(|s|+1)^{-\lambda}/M$ using 
$\ll p_1^{\lceil a_1/3\rceil}\cdots p_h^{\lceil a_h/3\rceil}\, \nu_2^{\kappa_9}$ operations
on numbers of $\ll \nu_2^2$ bits.
Since there are $\le |\mathcal{I}|\,(|\mathcal{V}|+1)\,(J+1) \ll (|s|+1)^{1/3} \nu_2^2$
such sums to be computed, then, on accounting for
all the truncation and finite precision errors introduced so far, we see that $L(s,\chi)$ can 
be computed to within $\pm\, q^{-\lambda}(|s|+1)^{-\lambda}$ using a further 
$\ll p_1^{\lceil a_1/3\rceil}\cdots p_h^{\lceil a_h/3\rceil} (|s|+1)^{1/3} \nu_2^{\kappa_9+2}$
operations. The theorem follows. 
\end{proof}

In anticipation of a practical implementation of the algorithm, let us
make a few more comments. 
In order to improve efficiency,  
one could assume a type of pseudo-randomness in the round-off errors
that accumulate from, say, summing a large number of terms; e.g.\ the 
sum over $\mathcal{V}$ in \eqref{tailsum}.
For example, one might model the round-off errors by a
sequence of independent identically distributed
random variables with mean zero, which therefore gives square-root 
cancellation in the aggregate error; see \cite{O2}. This way, 
the aggregate error is bounded in the $l^2$-norm (root-mean-square) rather than in the $l^1$-norm. 
Such a model is suitable in large-scale computations that focus on
statistics of zeros (e.g. \cite{O2}, \cite{Go}, and \cite{H3})
because it is robust, it increases the practical efficiency noticeably, and the $l^1$-error
obtained without assuming it might still suffice for many purposes such as 
verifying the Riemann hypothesis at relatively low height, or 
computing moments or zero statistics.
However, when checking the Riemann hypothesis in neighborhoods of very close zeros,
it might appear risky to rely on a model that 
assumes square-root cancellation  in the round-off errors 
since the Riemann hypothesis is itself, essentially, about square-root
cancellation (even though 
the cancellation in each situation occurs for different reasons).
Therefore, it is useful to have at least one algorithm implementation 
that controls the aggregate round-off error in the $l^1$--norm while also minimizing
the use of multi-precision arithmetic so as to avoid unnecessary increases in
the running time. Such an implementation can be carried out
with the aid of multi-precision packages (like \verb!MPFR! and \verb!GMP!).
%Another check of computational correctness is 
% replication via independent implementations.

\section{Comments on the general modulus case} \label{sec:general sum}

While the algorithm presented here for computing $S_{\chi}(K)$ does not yield a power-saving beyond $q^{1/2}$
when $q\in\{p,p^2,p^4\}$, and in fact it requires $q^{1+o(1)}$ time in the case $q=p$,
it is still consistent with the existence of a general $q^{1/3+o_{\lambda}(1)}$ algorithm
for computing character sums.
In the $t$-aspect (i.e.\ with $s=\sigma+it$ and thinking of $t$ large), there 
exists such an algorithm, as well as a faster one 
performing in $t^{4/13+o_{\lambda}(1)}$ time that relies on computing cubic
exponential sums; see~\cite{H2}. The similarities between
the algorithms in the $t$ and $q$ aspects rest heavily on the
power-full structure of the modulus, which suggests that 
in order to tackle the prime modulus case (and, likely, the square-free case)
we will need significant additional algorithms. 

In the remainder of this section, we give a
 general method for computing $S_{\chi}(K)$, where $\chi\bmod{q}$ is any character, 
and with no assumption about the factorization of $q$.
This method is primarily of interest in 
the range $q^{1/2} < K < q$. (One can always reduce to the range $K<q$ by the 
periodicity of $\chi$ and the observation $\sum_{0\le n<q} \chi(n) = 0$ if
$\chi$ is nonprincipal.)
It might be illuminating, though, to first consider the following more general situation. 
Let $\alpha_0,\ldots,\alpha_{R-1}$ be any sequence of numbers. 
Then, under some favorable conditions on $\alpha_r$,
 we describe a procedure for computing $\sum_{r<L} \alpha_r$
that can be faster than a straightforward evaluation
when $R^{1/2}<L<R$. To this end, 
note that the domain of definition of $\alpha_m$ can be extended 
to all $m\in \mathbb{Z}$ by setting $\alpha_m:=\alpha_{m \bmod{R}}$, and 
let $\hat{\alpha}_0,\ldots,\hat{\alpha}_{R-1}$ 
denote the dual sequence under the discrete Fourier transform, so
\begin{equation}
\hat{\alpha}_m:= \sum_{0\le r < R} \alpha_r e^{\frac{2\pi i m r}{R}}\,.
\end{equation}   
Then we have the following functional equation, valid for
$W$ in Schwartz class, say,
\begin{equation} \label{eq:ft}
\sum_{m=-\infty}^{\infty} \alpha_m\, W(m) = \frac{1}{R}\,\sum_{m=-\infty}^{\infty} \hat{\alpha}_m\, \hat{W}\left(\frac{m}{R}\right)\,,
\end{equation}
where $\hat{W}(x) := \int_{\mathbb{R}} W(y)e^{-2\pi ixy}\,dy$.

To compute $\sum_{r<L} \alpha_r$, we choose $W:=I*H$, 
where $I$ is the indicator function of $[0,L]$, and $H(y):=e^{-\pi y^2/R}$.
Notice that $I(y)$ is very well-approximated by $W(y)$, to within $\pm R^{-\lambda}$ say, except for two intervals
of length $R^{1/2+o_{\lambda}(1)}$ near $y=0$ and $y=L$.
So the difference between the l.h.s. 
of \eqref{eq:ft} and $\sum_{r<L} \alpha_r$ can be computed to within $LR^{-\lambda}\ll R^{-\lambda+1}$ 
using a sum of length $R^{1/2+o_{\lambda}(1)}$ terms.
Also, $\hat{W}(y/R)$ decays, with $y$, like $e^{-\pi y^2/R}$, which implies that the r.h.s of \eqref{eq:ft} can 
be made of length $R^{1/2+o_{\lambda}(1)}$ by truncating the series with an error $\ll R^{-\lambda}$. 
Letting $R^{\epsilon_0}$ denote the cost 
of computing an individual point $\alpha_j$, 
and letting $R^{\delta_0}$ denote the cost of computing 
an individual dual point $\hat{\alpha}_j$, 
we see that $\sum_{r<L} \alpha_r$ (and more generally $\sum_{L_0\le r<L_0+L}\alpha_r$)
can be computed, via \eqref{eq:ft}, in about
$R^{1/2+\delta_0+o_{\lambda}(1)}+R^{1/2+\epsilon_0+o_{\lambda}(1)}$ time,
instead of $L^{1+\epsilon_0+o_{\lambda}(1)}$ time.

In the case of 
a primitive Dirichlet character $\chi\bmod{q}$, we take $R=q$, so 
the dual is $\hat{\chi}(m) = \overline{\chi}(m)\,\tau(\chi)$.\footnote{If $\chi$ 
is not primitive, say it is induced by $\chi_1\bmod q_1$,
then, for $(m,q)=1$, we have
$\hat{\chi}(m) = \overline{\chi}(m)\chi_1(q/q_1)\mu(q/q_1)\tau(\chi_1)$
where $\mu$ is the M\"obius function.} 
%Notice that, if $\chi$ is real then the non-vanishing of $\hat{\chi}$ tells 
%whether $q$ is square-free without explicitly factoring $q$.}
If, in addition, $\chi$ is real, then
we have a simple formula for the Gauss sum $\tau(\chi)$,
and so it is easy to see that $\hat{\chi}$ can be computed 
provably quickly (in poly-log time) by appealing to quadratic reciprocity.
We conclude that $\sum_{r<L}\chi(r)$ can always 
be computed in $q^{1/2+o_{\lambda}(1)}$ time for real primitive $\chi$.
In the case of a general character, we do not have quadratic reciprocity, but
we  can still express the Gauss sum as a ratio
of series involving $q^{1/2+o_{\lambda}(1)}$ terms by applying formula \eqref{eq:ft}
with $\alpha_n = \chi(n)$ and with $W(x)$ a self-similar function with sufficient decay (e.g. $e^{-\pi x^2/q}$). 

We also mention the following. Let $p$ be an odd prime, let
$\chi$ be the real primitive character $\textrm{mod}\,\,p$, and let
\begin{eqnarray}\label{eq:tdef}
g(x):=\frac{1}{\sqrt{p}}\,\sum_{0\le k <p} e^{\frac{2\pi i x k^2}{p}}\,.
\end{eqnarray} 
It is well-known that if $\chi$ is even and $(m,p)=1$, then $\chi(m)= g(m)$. 
Similarly, if $\chi$ is odd and $(m,p)=1$, then $\chi(m) = -ig(m)$. 
Assuming that $\chi$ is even, say, we extend the domain of definition of $\chi$ 
to all of $x\in\mathbb{R}$, setting $\chi(x):=g(x)$. 
By quadratic reciprocity, $\chi(x)$ can be computed, for integer $x$, 
in poly-log time.
A consequence of the algorithm  
for computing quadratic exponential sums in Theorem~\ref{thetaalg} 
is that one can still compute $\chi(x)$ in poly-log time for any $x\in\mathbb{R}$.

\vspace{3mm}

\noindent
\textbf{Acknowledgements.} I would like to thank the referee for many useful comments
and for pointing out the reference \cite{BLT}, 
Professor Peter Sarnak for suggesting the problem of computing character sums to me
and for the reference \cite{GL}, Professor Henri Cohen for a useful comment about Gauss sums, 
and Professor Roger Heath-Brown for mentioning the reference \cite{I}. I also would like
to thank Michael Rubinstein and Pankaj Vishe for related conversations.

\end{document}